\documentclass[11pt]{amsart}
\usepackage{xcolor}
\usepackage[a4paper,hmargin=3cm,vmargin=3cm]{geometry}
\usepackage{amsfonts,amssymb,amscd,amstext,hyperref}
\usepackage{graphicx}
\usepackage{color}
\usepackage[dvips]{epsfig}
\usepackage[utf8]{inputenc}
\usepackage{comment}

\usepackage{fancyhdr}
\usepackage{times}

\usepackage{enumerate}
\usepackage{titlesec}
\usepackage{mathrsfs}
\usepackage{stmaryrd}

\def\cD{\mathcal{D}}

\let\hat=\widehat
\let\tilde=\widetilde

\newfont{\bb}{msbm10 at 12pt}

\titleformat{\section}
{\filcenter\bfseries\large} {\thesection{.}}{0.2cm}{}
\titleformat{\subsection}[runin]
{\bfseries} {\thesubsection{.}}{0.15cm}{}[.]
\titleformat{\subsubsection}[runin]
{\em}{\thesubsubsection{.}}{0.15cm}{}[.]

\usepackage[up,bf]{caption}

\newtheorem{theorem}{Theorem}[section]
\newtheorem{lemma}[theorem]{Lemma}
\newtheorem{proposition}[theorem]{Proposition}

\newtheorem{remark}[theorem]{Remark}

\newtheorem{definition}[theorem]{Definition}

\theoremstyle{definition}

\numberwithin{equation}{section}
\numberwithin{figure}{section}

\usepackage{color}

\begin{document}
\fancyhead[LO]{Free boundary saddle disks}
\fancyhead[RE]{Alberto Cerezo}
\fancyhead[RO,LE]{\thepage}

\thispagestyle{empty}

\begin{center}
{\bf \LARGE Free boundary saddle disks in the unit ball}

\vspace*{5mm}

\hspace{0.2cm} {\Large Alberto Cerezo}
\end{center}

\vspace{0.5cm}


\footnote[0]{
\noindent \emph{Mathematics Subject Classification}: 53A10, 53C42. \\ \mbox{} \hspace{0.25cm} \emph{Keywords}: saddle surfaces, free boundary}

\vspace*{7mm}

\begin{quote}
{\small
\noindent {\bf Abstract.}\hspace*{0.1cm}
We construct an infinite family of non-planar free boundary disks of non-positive Gaussian curvature in the unit ball of $\mathbb{R}^3$.
\vspace*{0.1cm}
}
\end{quote}

\section{Introduction}

A classical result of Almgren \cite{A}, also known as Calabi-Almgren theorem, asserts that any minimal 2-sphere in $\mathbb{S}^3$ must be a totally geodesic equator; see also \cite{C}. Almgren's proof relies on the use of the Hopf differential, a holomorphic quadratic differential associated to any constant mean curvature immersion in a space form, which vanishes precisely when the surface is totally umbilical.

The problem of understanding to what extent Almgren's uniqueness theorem extends to more general classes of surfaces was recently studied by Gálvez, Mira and Tassi \cite{GMT}. More specifically, they proved that if $\Sigma \subset \mathbb{S}^3$ is an analytic {\em saddle} sphere in $\mathbb{S}^3$, then $\Sigma$ must be a totally geodesic equator. The {\em saddle} condition here means that the product $\kappa_1\kappa_2$ of the principal curvatures of $\Sigma$ is non-positive. They also constructed $\mathcal{C}^\infty$ saddle spheres in $\mathbb{S}^3$ different from equators, showing that the analyticity hypothesis is necessary to ensure uniqueness. Other examples of non-analytic saddle spheres were previously constructed by Martinez-Maure \cite{MM} and Panina \cite{P}.

In the context of surfaces with boundary, the natural analogue of Almgren's theorem is due to Nitsche \cite{N}, who proved in 1985 that any free boundary minimal disk in the Euclidean unit ball $\mathbb{B}^3$ must be equatorial; see Figure \ref{fig:FBdisk}. We recall that a compact orientable surface $\Sigma$ with boundary $\partial \Sigma$ immersed in $\mathbb{B}^3$ is said to be \emph{free boundary} in the unit ball if $\partial \Sigma \subset \partial \mathbb{B}^3$ and $\Sigma$ meets $\partial \mathbb{B}^3$ orthogonally along $\partial \Sigma$. This condition appears naturally when one considers critical points for the area functional among all surfaces $\Sigma$ with boundary in $\partial \mathbb{B}^3$. 

\begin{figure}[ht]
\centering
\includegraphics[width=0.4\textwidth]{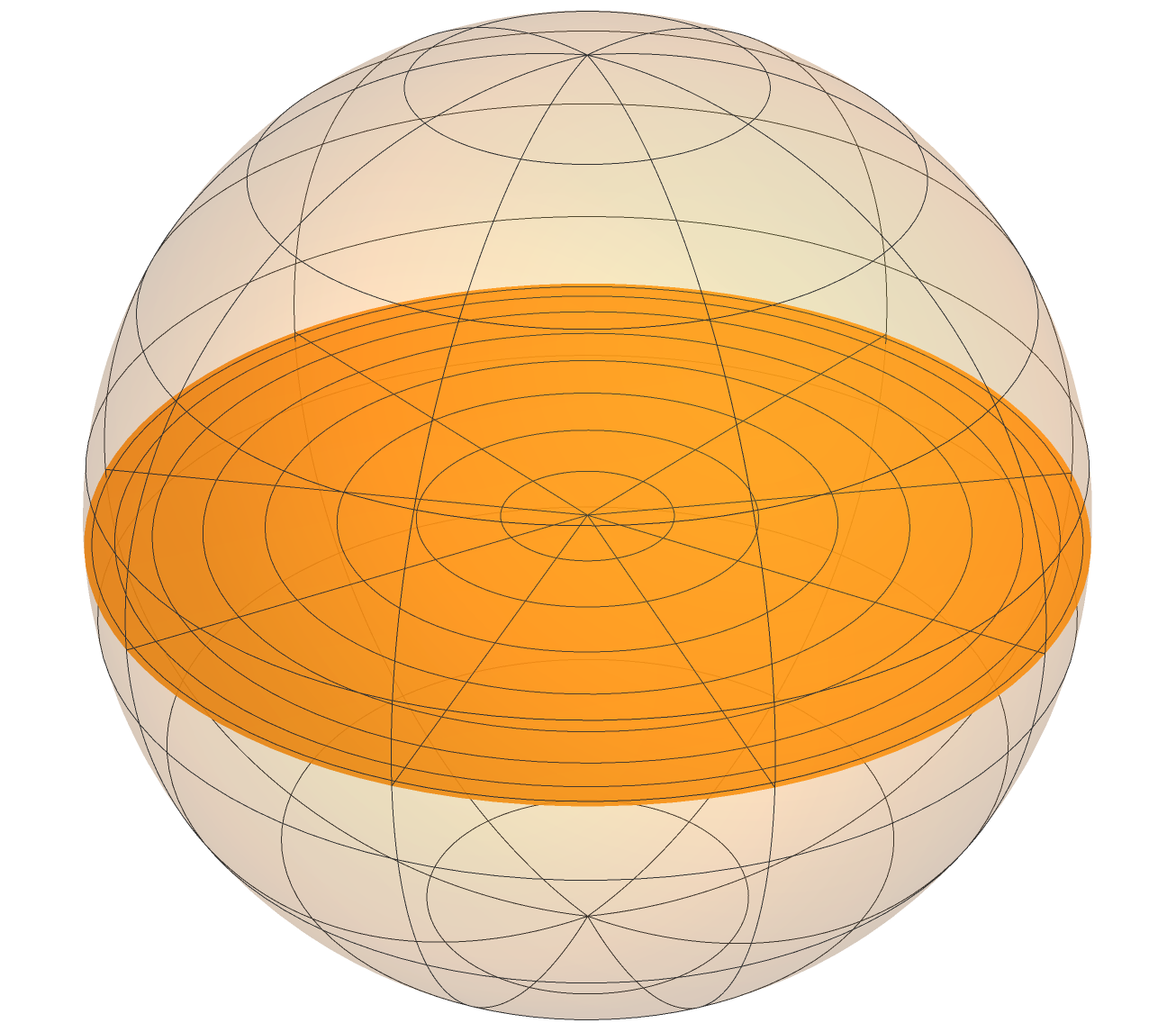}
\caption{Equatorial disk in $\mathbb{B}^3$.}
\label{fig:FBdisk}
\end{figure}

The idea behind Nitsche's theorem can be viewed as the free boundary version of Almgren’s uniqueness result. Indeed, assume by contradiction that $\Sigma$ is a free boundary minimal disk which is not totally geodesic, and consider the pair of line fields $L_1,L_2$ associated to the set of principal directions of the surface. These line fields are defined on the set of non-umbilical points of $\Sigma$. Since the Hopf differential of $\Sigma$ is holomorphic, it follows that the set of umbilical points of $\Sigma$ is discrete, and the index of the line fields is negative at each such point. On the other hand, the free boundary condition implies that at each non-umbilical point of $\partial \Sigma$, one of the line fields $L_1,L_2$ is tangent to $\partial \Sigma$ and the other one is orthogonal. We may then reflect $L_1,L_2$ by symmetry along the boundary. This implies the existence of a pair of line fields on a 2-sphere with a discrete set of singularities of negative index. This contradicts Poincaré-Hopf theorem \cite{Hop}. Hence, any free boundary disk $\Sigma$ in $\mathbb{B}^3$ must be totally geodesic.

It is possible to extend Nitsche's uniqueness result to broader classes of surfaces. Indeed, one can show that any free boundary disk $\Sigma$ modelled by a geometric PDE of the type
\begin{equation}\label{eq:Weingarten}
    W(\kappa_1,\kappa_2) = 0, \; \; \; \text{ with }  \; \; \; W(0,0) = 0  \; \; \;  \text{ and }  \; \; \;  W_{\kappa_1}W_{\kappa_2} > 0
\end{equation}
for some differentiable symmetric function $W(x,y)$ (i.e., $W(x,y) = W(y,x)$) must be totally geodesic. The surfaces satisfying \eqref{eq:Weingarten} are also known as {\em special elliptic Weingarten surfaces of minimal type}; see for example \cite{FGM,FM,RSa,SaT}. We note that minimal surfaces are special Weingarten surfaces in the particular case $W(\kappa_1,\kappa_2) = \kappa_1 + \kappa_2$. We also remark that every surface satisfying \eqref{eq:Weingarten} must be saddle.

In view of these results, it is natural to ponder whether a free boundary analogue of the result by Gálvez-Mira-Tassi \cite{GMT} holds. Specifically, we ask the following:

\begin{enumerate}[(1)]
\item Is every free boundary analytic saddle disk in $\mathbb{B}^3$ equatorial?
\item Are there non-equatorial free boundary saddle disks in $\mathbb{B}^3$ of class $\mathcal{C}^\infty$?
\end{enumerate}

An affirmative answer to both these problems would provide a sharp geometric extension of Nitsche's uniqueness theorem. In other words, we could view the class of free boundary analytic saddle disks as the {\em largest} family of surfaces in which Nitsche's result holds.

Regarding the uniqueness in the analytic case, the results obtained in \cite{GMT} could allow us to control the index of the line fields associated to the principal directions at the interior umbilical points of a free boundary saddle disk. However, the analysis of the index on the boundary is substantially more involved, and new techniques would be needed to deal with this situation.

On the other hand, the problem of whether there exist non-equatorial free boundary saddle disks $\Sigma$ in $\mathbb{B}^3$ of class $\mathcal{C}^\infty$ presents several difficulties. The examples constructed in \cite{GMT,MM,P} do not admit a natural free boundary analogue. We also note that any free boundary saddle disk must satisfy certain geometric restrictions. For example, by the Poincaré–Hopf theorem they must have umbilical points, but an analysis of the index at such points determines that this set cannot just consist of isolated umbilical points in the interior of $\Sigma$; see \cite{Voss}. In other words, any free boundary saddle disk should have umbilical points with positive index in the boundary or a non-discrete set of interior umbilical points. We prove the following:

\begin{theorem}\label{thm:main}
There exist infinitely many $\mathcal{C}^\infty$ free boundary saddle disks in the unit ball $\mathbb{B}^3$ which are not equatorial.
\end{theorem}

\begin{figure}[ht]
\centering
\includegraphics[width=0.4\textwidth]{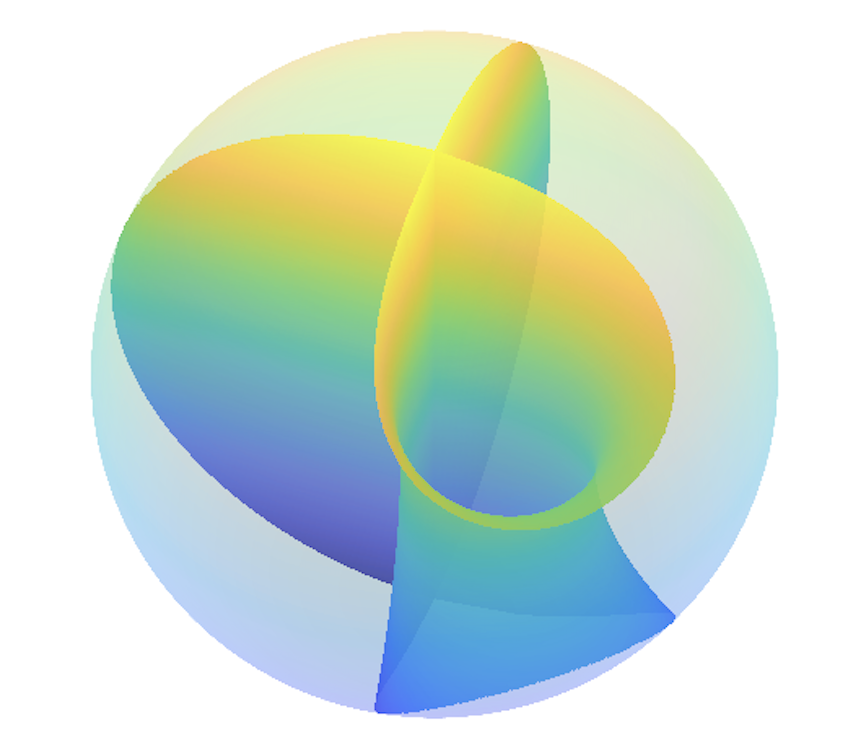}
\caption{One of the free boundary saddle disks in Theorem \ref{thm:main}.}
\label{fig:psigamma}
\end{figure}

We emphasize that the constructed examples are never analytic. We do not know whether analytic free boundary saddle disks other than the equatorial ones exist. In light of the results by Gálvez, Mira and Tassi \cite{GMT}, it is natural to conjecture that uniqueness holds when we restrict to the family of analytic examples. We also note that all the disks obtained in Theorem \ref{thm:main} possess self-intersections; see Figure \ref{fig:psigamma}. It remains an open question whether embedded free boundary saddle disks in $\mathbb{B}^3$ must necessarily be equatorial, regardless of their regularity. Finally, since the examples constructed here contain planar pieces, we also ask whether every free boundary saddle disk in $\mathbb{B}^3$ must have one such portion.

We briefly outline the structure of the paper. In Section \ref{sec:2}, we will construct a family of free boundary disks in $\mathbb{B}^3$. These surfaces are foliated by circles, and all of them contain a ribbon-shaped planar curve $\gamma$; see Definition \ref{def:ribbon}. Later on, in Section \ref{sec:saddledisks}, we will show that, under additional hypotheses on the ribbon curve $\gamma$, the associated free boundary disks are saddle.

\subsection*{Acknowledgements}
This research has been financially supported by Grants PID2020-118137GB-I00 and PID2024-160586NB-I00 funded by MICIU/AEI/10.13039/501100011033 and by ESF+.

\section{Construction of free boundary disks in \texorpdfstring{$\mathbb{B}^3$}{}}\label{sec:2}
We will introduce a method to construct a free boundary immersion $\psi(u,v)$ of a disk in the unit ball $\mathbb{B}^3 \subset \mathbb{R}^3$ in terms of a certain ribbon-shaped planar curve $\gamma(u)$, which we describe next. In what follows, let $\mathbb{D}^2 = \mathbb{B}^3 \cap \{x_3 = 0\}$ denote the horizontal unit disk in $\mathbb{B}^3$.

\begin{definition}\label{def:ribbon}
    Let $\gamma$ be a planar regular curve in $\mathbb{D}^2$. We say that $\gamma$ is a {\em ribbon} if, up to an isometry, it admits a parametrization $\gamma(u):[-u_3,u_3] \to \mathbb{D}^2$, where $u$ is the arc-length parameter, satisfying the following properties:
    \begin{enumerate}
    \item $\Psi(\gamma(u)) = \gamma(-u)$, where $\Psi$ denotes the symmetry with respect to the line $\{x = y \}$. In particular, $\gamma(0)$ belongs to this line.
    \item There exist $a \in (0,1)$ and $u_2 \geq 0$ such that, on the interval $[u_2,u_3]$, $\gamma(u)$ parametrizes the segment $\{0\}\times [-1,a] \times \{0\}$, with $\gamma(u_3) = (0,-1,0)$. In particular, $u_2 = u_3 - (1+a) > 0.$ Consequently, on the interval $[-u_3,-u_2]$, $\gamma(u)$ parametrizes the segment $[-1,a]\times \{0\} \times \{0\}$, and $\gamma(-u_3) = (-1,0,0)$.
    \item On the interval $(-u_2,u_2)$, $\gamma(u)$ is an embedded, strictly convex curve contained in the open planar quadrant $\left((0,1)\times(0,1)\times \{0\}\right) \cap \mathbb{D}^2$; see Figure \ref{fig:gamma}. The rotation index of the curve is thus $3/4$. 
    \item The curvature $\kappa(u)$ of $\gamma(u)$ is strictly decreasing on the interval $(u_1,u_2)$ for some $u_1 \in (0,u_2)$. By symmetry, $\kappa(u)$ is increasing for $u \in (-u_2,-u_1)$.
    \end{enumerate}
\end{definition}

\begin{figure}[ht]
\centering
\includegraphics[width=0.4\textwidth]{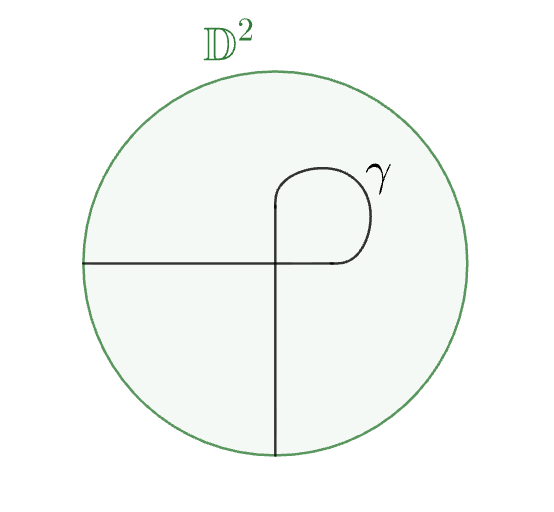}
\caption{Example of a ribbon $\gamma$.}
\label{fig:gamma}
\end{figure}

\begin{remark}
    There exist infinitely many $\mathcal{C}^\infty$ curves in $\mathbb{D}^2$ satisfying the properties listed above. However, none of these curves is analytic.
\end{remark}

\begin{lemma}\label{lem:linindep}
The vectors $\gamma(u)$ and $\gamma'(u)$ are linearly independent for $u \in (-u_2,u_2)$.
\end{lemma}
\begin{proof}
    Due to the symmetry of $\gamma$ with respect to the line $\{x = y\}$, it suffices to show the result for $u \in (-u_2,0]$. We can express the unit tangent vector $\gamma'(u)$ as $(\cos\theta(u),\sin\theta(u))$ for some function $\theta = \theta(u)$. By definition of ribbon, it holds
    $$\theta(-u_2) = 0, \; \; \; \theta(0) = \frac{3\pi}{4}.$$ 
    Moreover, since $\gamma$ is convex, the function $\theta(u)$ must be strictly increasing, so we may reparametrize this curve in terms of $\theta \in [0,\frac{3\pi}{4}]$. We have then
    $$\gamma(\theta) = \left(a + \int_0^{\theta}f(\nu)\cos\nu d\nu,\int_0^{\theta}f(\nu)\sin\nu d\nu\right),$$
    where $f(\nu) := \frac{\partial u}{\partial \theta} > 0$. Now, the condition for $\gamma(\theta)$ and $\gamma'(\theta)$ to be linearly independent is that
    $$F(\theta) = \sin\theta\left(a + \int_0^{\theta}f(\nu)\cos\nu d\nu\right) - \cos\theta\int_0^{\theta}f(\nu)\sin\nu d\nu$$
    does not vanish for any $\theta \in \left(0,\frac{3\pi}{4}\right)$. If $\theta \in \left(0,\frac{\pi}{2}\right)$, it is clear that $F(\theta) > 0$, as $F(0) = 0$ but
    $$F'(\theta) = \cos\theta\left(a + \int_0^{\theta}f(\nu)\cos\nu d\nu\right) + \sin\theta\int_0^{\theta}f(\nu)\sin\nu d\nu > 0.$$
    Now, if $\theta \in \left(\frac{\pi}{2},\frac{3\pi}{4}\right]$, we have that $\gamma'(\theta)$ belongs to the quadrant $[-1,0)\times [0,1)\times \{0\}$, while $\gamma(\theta)$ belongs to $(0,1)\times (0,1)\times \{0\}$. Hence, $\gamma(\theta)$ and $\gamma'(\theta)$ must be linearly independent in this case as well. This completes the proof of the lemma.
\end{proof}

\begin{definition}\label{def:cpv}
    Given $p \in \mathbb{D}^2$, $p \neq (0,0,0)$, let $\Pi_p$ be the vertical plane passing through the origin and $p \in \mathbb{D}^2$. We then define the curve $c_p \subset \Pi_p$ as the unique circle passing through $p$ which intersects $\mathbb{D}^2$ and $\partial \mathbb{B}^3$ orthogonally (see Figure \ref{fig:cpv}). If $p = (0,0,0)$, we denote by $c_p$ the vertical segment $\{0\}\times \{0\}\times [-1,1]$.
\end{definition}

\begin{figure}[ht]
\centering
\includegraphics[width=0.4\textwidth]{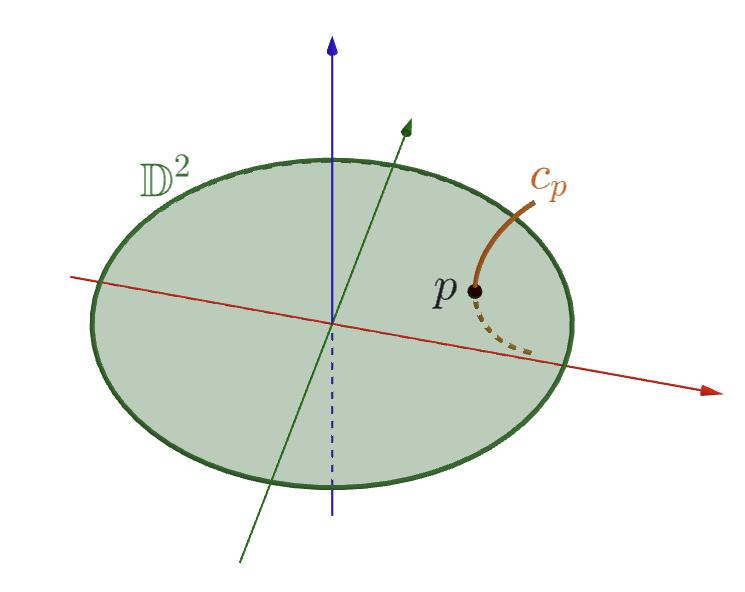}
\caption{Piece of a curve $c_p$ inside $\mathbb{B}^3$.}
\label{fig:cpv}
\end{figure}

\begin{remark}\label{rem:cpvparametrization}
    Let $p \in \mathbb{D}^2$ and $v$ be the arc-length parameter of the curve $c_p = c_p(v)$, so that $c_p(0) = p$ and $c_p'(0) = (0,0,1)$. For every $p$, there exists an interval $I_p := [-v_p,v_p]$, $v_p > 0$, such that $c_p(v) \in \mathbb{B}^3$ for all $v \in I_p$, and $c_p(-v_p),c_p(v_p) \in \partial \mathbb{B}^3$. In particular, if $p \in \partial \mathbb{D}^2$, $I_p = \{0\}$.
\end{remark}

\begin{definition}\label{def:psigamma}
   Given a ribbon $\gamma(u)$ in $\mathbb{D}^2$, we define $\Sigma_\gamma$ as the surface given by the immersion $\psi_\gamma(u,v):\cD \to \mathbb{B}^3$, where
   $$\psi_\gamma(u,v) := c_{\gamma(u)}(v),$$
   and $\cD \subset \mathbb{R}^2$ is the disk
   \begin{equation}\label{eq:cD}
   \cD := \{(u,v) : u \in [-u_3,u_3], v \in I_{\gamma(u)}\}.
   \end{equation}
   Here, $I_{\gamma(u)}$ denotes the interval introduced in Remark \ref{rem:cpvparametrization}; see Figure \ref{fig:psigamma}.
\end{definition}

\begin{remark}\label{rem:Gaussmap}
    For every $u$ such that $\gamma(u) \neq (0,0)$, we can explicitly compute $\psi_\gamma(u,v)$ as
    \begin{equation}\label{eq:psigamma}
        \psi_\gamma(u,v) = \left(\gamma_1(u) F_{\gamma}(u,v),\gamma_2(u) F_{\gamma}(u,v), R_{\gamma}(u) \sin\left(\frac{v}{R_{\gamma}(u)}\right)\right),
    \end{equation}
    where
    \begin{equation}\label{eq:RF}
    \begin{aligned}
        R_{\gamma}(u) &:= \frac{1}{2}\left(\frac{1}{\|\gamma(u)\|} - \|\gamma(u)\|\right), \\
        F_{\gamma}(u,v) &:= 1 + \frac{R_{\gamma}(u)}{\|\gamma(u)\|}\left(1 - \cos\left(\frac{v}{R_{\gamma}(u)}\right)\right).
    \end{aligned}
    \end{equation}
    Moreover, we choose the Gauss map $N_\gamma(u,v)$ of $\psi_\gamma$ so that the frame $\{(\psi_\gamma)_u,(\psi_\gamma)_v,N_\gamma\}$ is positively oriented.
\end{remark}

\begin{remark}\label{rem:cD}
    We emphasize that $\cD$ in \eqref{eq:cD} is a disk: indeed, if $u = \pm u_3$, then $\gamma(u) \in \partial \mathbb{D}^2$, and so the interval $I_{\gamma(u)}$ reduces to $\{0\}$. Moreover, $\psi_\gamma(u,v)$ will be of class $\mathcal{C}^\infty$ if $\gamma(u)$ is.
\end{remark}

\begin{proposition}\label{pro:psiFB}
    Let $\gamma(u)$ be a ribbon and $\psi_\gamma(u,v)$ be the corresponding immersion in Definition \ref{def:psigamma}. Then, $\Sigma_\gamma$ is a free boundary disk in $\mathbb{B}^3$.
\end{proposition}
\begin{proof}
    We may decompose $\Sigma_\gamma$ into three parts, namely
    \begin{equation}\label{eq:Sigma123}
    \begin{aligned}
        \Sigma_{\gamma}^{(1)} &:= \psi_\gamma\left(\mathcal{D} \cap \{u \leq -u_2\}\right), \\
        \Sigma_{\gamma}^{(2)} &:= \psi_\gamma\left(\mathcal{D} \cap \{-u_2 < u < u_2\}\right), \\
        \Sigma_{\gamma}^{(3)} &:= \psi_\gamma\left(\mathcal{D} \cap \{u \geq u_2\}\right),
        \end{aligned}
    \end{equation}
    respectively. It is clear that $\Sigma_{\gamma}^{(1)}$ and $\Sigma_{\gamma}^{(3)}$ are pieces of equatorial disks, so they meet $\partial \mathbb{B}^3$ orthogonally; see Definition \ref{def:ribbon}. We also note that the piece $\Sigma_{\gamma}^{(2)}$ intersects $\partial \mathbb{B}^3$ along the curves $u \mapsto \psi_{\gamma}(u,\pm v_{\gamma(u)})$, $u \in (-u_2,u_2)$. For every $u_0 \in (-u_2,u_2)$, the circle $v \mapsto \psi_\gamma(u_0,v) = c_{\gamma(u_0)}(v)$ meets $\partial \mathbb{B}^3$ orthogonally; see Definition \ref{def:cpv}. Thus, $\Sigma_{\gamma}^{(2)}$ meets $\partial \mathbb{B}^3$ orthogonally, as we wanted to show.
\end{proof}

\section{The saddle condition}\label{sec:saddledisks}

Let $\gamma$ be a ribbon and $\Sigma_\gamma$ be its associated free boundary disk. In general, $\Sigma_\gamma$ will not be a saddle surface. However, in this section we will prove that we can define a new ribbon $\tilde \gamma$ in terms of $\gamma$ such that the new free boundary disk $\Sigma_{\tilde \gamma}$ is actually saddle. 

Arguing as in the proof of Proposition \ref{pro:psiFB}, we can decompose every surface $\Sigma_\gamma$ into three pieces; see \eqref{eq:Sigma123}. The surfaces $\Sigma_{\gamma}^{(1)}$ and $\Sigma_{\gamma}^{(3)}$ are pieces of planes, so they are trivially saddle. Hence, it remains to study the curvature of the piece $\Sigma_{\gamma}^{(2)}$. We first prove the following:

\begin{proposition}\label{pro:curvaturapositiva}
    Let $\gamma(u)$ be a ribbon. For any $u_0 \in (-u_2,u_2)$, the normal curvature of the vertical circle $v \mapsto \psi_\gamma(u_0,v) = c_{\gamma(u_0)}(v)$ in $\Sigma_\gamma$ is positive.
\end{proposition}
\begin{proof}
    Let $\Pi_{u_0}$ denote the vertical plane in which $c_{\gamma(u_0)}$ is contained, and denote by $n = n(u_0,v)$ the inner normal vector to $c_{\gamma(u_0)}(v)$ in $\Pi_{u_0}$. To prove the proposition, it suffices to show that the projection $\hat N(u_0,v)$ of the Gauss map $N(u_0,v)$ of $\Sigma_\gamma$ onto $\Pi_{u_0}$ is a non-vanishing vector which has the same direction as $n(u_0,v)$: indeed, since $\psi_\gamma(u_0,v) = c_{\gamma(u_0)}(v)$ is a circle, this curve necessarily has positive normal curvature in $\Sigma_\gamma$.

    Assume first that $u_0 = 0$. Due to the symmetry properties of $\gamma(u)$ (see Definition \ref{def:ribbon}), it follows that $\Sigma_\gamma$ is symmetric with respect to the vertical plane $\Pi_0 = \{x = y\} \subset \mathbb{R}^3$. Thus, the Gauss map $N(0,v)$ lies in $\Pi_0$, i.e., $\hat N(0,v) = N(0,v)$. In fact, $N(0,v)$ must coincide with $n(0,v)$ due to the orientation chosen for this map; see Remark \ref{rem:Gaussmap}.

    Let us assume then that $u_0 \neq 0$. Since the projection $\hat N(u_0,v)$ must be orthogonal to $v \mapsto \psi_\gamma(u_0,v)$, it follows that $\hat N(u_0,v)$ and $n(u_0,v)$ must be collinear. Now, as we already proved for $u = 0$ that $\hat N(0,v)$ and $n(0,v)$ coincide, then necessarily $\hat N(u_0,v)$ and $n(u_0,v)$ must have the same direction as long as $\hat N(u_0,v)$ does not vanish anywhere. This condition is satisfied: indeed, assume by contradiction that $\hat N(u_0,v)$ vanishes for some $(u_0,v)$. Geometrically, this means that $N(u_0,v)$ is orthogonal to the vertical plane $\Pi_{u_0}$, so in particular $\left(\psi_\gamma\right)_u(u_0,v) \in \Pi_{u_0}$. Hence, the horizontal components of $\psi_\gamma(u_0,v)$ and $(\psi_\gamma)_u(u_0,v)$ must be collinear. This is impossible: by \eqref{eq:psigamma}, it would follow that
    $$\left(\gamma_1'(u_0)F_\gamma(u_0,v) + \gamma_1(u_0)(F_\gamma)_u(u_0,v)\right)\gamma_2(u_0)F_\gamma(u_0,v) =$$ $$\left(\gamma_2'(u_0)F_\gamma(u_0,v) + \gamma_2(u_0)(F_\gamma)_u(u_0,v)\right)\gamma_1(u_0)F_\gamma(u_0,v),$$
    which reduces to $\gamma_1'(u_0) \gamma_2(u_0) = \gamma_1(u_0)\gamma_2'(u_0)$. This implies that $\gamma(u_0)$ and $\gamma'(u_0)$ are collinear at $u_0$, which is impossible by Lemma \ref{lem:linindep}. This completes the proof of the proposition.
\end{proof}

The previous proposition states that for every ribbon $\gamma$, the vertical circles $v \mapsto c_{\gamma(u_0)}(v)$, $u_0 \in (-u_2,u_2)$ that foliate the piece of surface $\Sigma_{\gamma}^{(2)}$ have positive normal curvature. Thus, if we prove that $\Sigma_{\gamma}^{(2)}$ admits another foliation by curves with negative normal curvature, it would follow that $\Sigma_{\gamma}^{(2)}$ is saddle. As we mentioned at the beginning of this section, not every surface $\Sigma_{\gamma}^{(2)}$ has this property. Nevertheless, we will be able to deform $\gamma$ into a new ribbon $\tilde \gamma$ in a way such that the piece $\Sigma_{\gamma}^{(2)}$ is saddle. We study this property next.

\begin{definition}\label{def:gammat}
    Let $\gamma(u):[-u_3,u_3] \to \mathbb{D}^2$ be a ribbon with the notations established in Definition \ref{def:ribbon}. For any $t \in (0,1]$, we define $\gamma_t(u):[-u_3,u_3] \to \mathbb{D}^2$ as the ribbon satisfying:
    \begin{enumerate}
        \item $\gamma_t$ parametrizes the segment $[-1,ta]\times \{0\} \times \{0\}$ on the interval $u \in [-u_3,-u_2]$.
        \item On the interval $u \in [u_2,u_3]$, $\gamma_t$ parametrizes the segment $\{0\}\times [-1,ta] \times \{0\}$.
        \item For $u \in (-u_2,u_2)$, $\gamma_t(u) = t\gamma(u)$.
    \end{enumerate}
\end{definition}

\begin{remark}
    Intuitively, $\gamma_t$ is the result of applying an homothety of factor $t \leq 1$ to the ribbon $\gamma$; see Figure \ref{fig:gammat}. Moreover, $\gamma_t$ is $\mathcal{C}^\infty$ if $\gamma$ is. We emphasize nevertheless that the induced disks $\Sigma_{\gamma_t}$ and $\Sigma_\gamma$ are not related by an homothety. We also note that unless $t = 1$, $\gamma_t(u)$ is not parametrized by arc length.
\end{remark}

\begin{figure}[ht]
\centering
\includegraphics[width=0.4\textwidth]{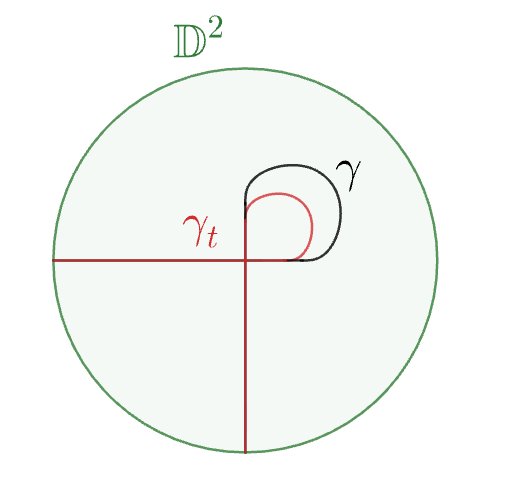}
\caption{Ribbon $\gamma_t$ for $t = 3/4$.}
\label{fig:gammat}
\end{figure}

\begin{proposition}\label{pro:curvaturanegativa}
Given a ribbon $\gamma$ there exists $t_0 \in (0,1]$ such that the surface $\Sigma_{\gamma_{t}}^{(2)}$ is saddle for every $t \in (0,t_0]$; see Definition \ref{def:gammat}. 
\end{proposition}
\begin{proof}
We will show that, for sufficiently small $t$, the horizontal level sets 
\begin{equation}\label{eq:Gtc}
    \Gamma_{(t,c)} := \Sigma_{\gamma_t}^{(2)} \cap \{x_3 = c\}
\end{equation}
with $|c| \leq 1$ are regular curves and have negative normal curvature. Along with Proposition \ref{pro:curvaturapositiva}, it follows that $\Sigma_{\gamma_t}^{(2)}$ is saddle.

    Let us first show that these horizontal level sets are regular curves. By definition of $\gamma_t$, it holds $\|\gamma_t(u)\| < t$. Thus, by \eqref{eq:RF}, it follows that $R_{\gamma_t}(u) > M$ for some function $M = M(t)$, where $M(t)$ diverges as $t$ approaches zero. In particular, choosing $t_0$ such that $M(t) > 1$ for every $t \leq t_0$ and using \eqref{eq:psigamma}, it is possible to deduce that the intersection in \eqref{eq:Gtc} is transversal, and in fact the corresponding level curve can be parametrized as
    \begin{equation}\label{eq:Gtcparam}
        \begin{aligned}
            \Gamma_{(t,c)}(u) &= \gamma_t(u)\left(1 + \frac{G_t(u)}{\|\gamma_t(u)\|}\right) + c {\bf e}_3, \\
            G_t(u) &:= R_{\gamma_t}(u)\left(1 - \sqrt{1 - \left(\frac{c}{R_{\gamma_t}(u)}\right)^2}\right),
        \end{aligned}
    \end{equation}
 where $\gamma_t \equiv (\gamma_{t,1},\gamma_{t,2},0)$, ${\bf e}_3 = (0,0,1)$, and $\gamma_{t,1}, \gamma_{t,2}$ denote the first and second components of $\gamma_t$. 
 
 We will now show that $\Gamma_{(t,c)}$ is convex for every $|c| \leq 1$. We can express $G_t(u)$ in \eqref{eq:Gtcparam} as
 $$G_t(u) = c^2\|\gamma(u)\|t + t^3\|\gamma(u)\|J_t(u)$$
 for some $J_t(u)$ which is analytic in $t \in [0,1]$ and $u \in [-u_2,u_2]$. In particular, it is bounded. Thus, by \eqref{eq:Gtcparam}, we deduce
 \begin{equation}\label{eq:Gtcderivadas}
     \begin{aligned}
      \Gamma_{(t,c)}'(u) &= t\gamma'(u)(1 + c^2 + t^2J_t(u)) + t^3\gamma(u)J_t'(u), \\
      \Gamma_{(t,c)}''(u) &= t\gamma''(u)\left(1 + c^2 + t^2 J_t(u)\right) + 2t^3\gamma'(u)J_t'(u) + t^3\gamma(u) J_t''(u).
     \end{aligned}
 \end{equation}
 Let $\hat n(u)$ be the unit normal vector to the planar curve $\Gamma_{(t,c)}(u)$ and denote by $n(u)$ the unit normal to $\gamma(u)$. We choose $n(u)$ so that the curvature $\kappa(u)$ of the convex curve $\gamma(u)$, given by $\kappa(u) = \langle n,\gamma''\rangle$, is positive. According to the first derivative in \eqref{eq:Gtcderivadas}, it holds that $\hat n(u)$ is collinear to $$n(u) + t^2\gamma'(u) \langle n(u),\gamma(u)\rangle K_t(u),$$ 
 for some bounded analytic function $K_t(u)$. Thus, we conclude that $\Gamma_{(t,c)}(u)$ will be convex if and only if the quantity
 $$\alpha(u) := \frac{1}{t}\langle \Gamma_{(t,c)}''(u), n(u) + t^2 \gamma'(u) \langle n(u),\gamma(u) \rangle K_t(u) \rangle$$
 is positive. By \eqref{eq:Gtcderivadas},
 \begin{equation}\label{eq:alpha}
 \begin{aligned}\alpha(u) &= \kappa (1 + c^2) + \kappa t^2 J_t + t^2 \langle n, \gamma\rangle J_t''  \\
 &+ 2t^4\langle n, \gamma \rangle J_t' K_t + t^4\langle\gamma , \gamma' \rangle \langle n, \gamma \rangle J_t''K_t.
 \end{aligned}
 \end{equation}
 We will now show that $\alpha(u) > 0$ if $t$ is sufficiently near zero. We will split the proof in two cases, depending on whether $u \in [0,u_1]$ or $u \in [u_1,u_2)$, where we use the notations in Definition \ref{def:ribbon}. The case $u \leq 0$ is immediate by the symmetry of $\gamma(u)$; see Definition \ref{def:ribbon}. Assume first that $u \in [0,u_1]$. We know that the curvature $\kappa(u)$ of $\gamma(u)$ is strictly positive on this closed interval; see Definition \ref{def:ribbon}. Thus, setting $\kappa_0 := \min_{u \in [0,u_1]}\kappa(u) > 0$, it is clear by \eqref{eq:alpha} that for $t$ small enough, the bound $\alpha(u) > \frac{1}{2}\kappa_0(1 + c^2) > 0$ is satisfied. 
 
 Let us then assume that $u \in [u_1,u_2)$. We will find a suitable bound for the product $\langle n(u), \gamma(u)\rangle$. We can express $n(u)$ and $\gamma(u)$ for $u < u_2$ as
\begin{align*}
 n(u) &= n(u_2) + \int_u^{u_2} \kappa(s)\gamma'(s)ds, \\
 \gamma(u) &= \gamma(u_2) - \int_u^{u_2}dw\left(\gamma'(u_2) - \int_{w}^{u_2}\kappa(z)n(z)dz\right).
\end{align*}
 By definition of ribbon, at $u = u_2$ it holds $\gamma(u_2) = (0,a,0)$, $\gamma'(u_2) = (0,-1,0)$ and $n(u_2) = (1,0,0)$. Hence,
 \begin{align*}
     |\langle n(u),\gamma(u)\rangle | & \leq \left|\int_u^{u_2} dw \int_w^{u_2} \kappa(z) \langle n(z), n(u_2)\rangle dz\right|  + \left|\int_u^{u_2} ds \int_u^{u_2}  \kappa(s) \langle \gamma'(s), \gamma'(u_2)\rangle dw\right|\\
     & + \left|\int_u^{u_2} \kappa(s) \langle \gamma'(s),\gamma(u_2)\rangle ds \right| + \left|\int_u^{u_2} ds \int_u^{u_2} dw \int_w^{u_2} \kappa(s)\kappa(z)\langle \gamma'(s), n(z)\rangle dz \right|.
 \end{align*}
By hypothesis, $\kappa(u)$ is positive and decreasing for $u \in [u_1,u_2]$; see Definition \ref{def:ribbon}. Hence, all of the above integrals can be bounded by $C\kappa(u)$ for some positive constant $C > 0$ independent of $u$. In particular, by \eqref{eq:alpha} and the boundedness of the functions $J_t$, $K_t$, we deduce that there exists a constant $C_1 > 0$ such that
$$\alpha(u) > \kappa(u) (1 + c^2) - \kappa(u) t^2 C_1 = \kappa(u)(1 + c^2 - t^2C_1).$$
Hence, we conclude that $\alpha(u) > 0$ for $u \in [u_1,u_2)$ and $t$ close enough to zero, showing that the planar curve $\Gamma_{(t,c)}(u)$ is convex.

So far, we have shown that the planar curves $\Gamma_{(t,c)}$ in \eqref{eq:Gtc} are convex for all $|c| \leq 1$ for small values of $t$. We will now show that the normal curvature of these curves in $\Sigma_{\gamma_t}^{(2)}$ is negative. To do so, it suffices to prove that the product $\langle \hat n, N \rangle$ is negative, where we recall that $\hat n$ denotes the unit normal to $\Gamma_{(t,c)}$ and $N$ is the Gauss map of $\Sigma_{\gamma_t}$. We will study the cases $c = 0$ and $c \neq 0$ separately. If $c = 0$, $\Gamma_{(t,0)}(u) \equiv \gamma_t(u)$, and so $\hat n(u)$ coincides with the normal vector $n(u)$ of $\gamma(u)$. This vector is precisely $-N(u,0)$, due to the choice of orientation for the Gauss map; see Remark \ref{rem:Gaussmap}. We will now consider the case $c \neq 0$. By continuity, we just need to prove that the product $\langle \hat n, N \rangle$ does not vanish. This is clear: indeed, the horizontal plane $\{x_3 = c\}$ is spanned by the tangent vector $\Gamma_{(t,c)}'$ and the normal unit vector $\hat n$. Moreover, it holds $\langle \Gamma_{(t,c)}', N\rangle = 0$. So, if $\langle N, \hat n\rangle$ were also zero at some point $(u_0,v_0)$, it would follow that $N(u_0,v_0)$ is vertical. This leads to a contradiction, as we chose $t$ small enough so that $\Sigma_t^{(2)}$ meets the plane $\{x_3 = c\}$ transversally. This shows that the level curves $\Gamma_{(t,c)}$ have negative normal curvature, and complete the proof of the proposition.
\end{proof}

\subsection*{Proof of Theorem \ref{thm:main}}
Let $\gamma(u)$ be any ribbon and consider the family of ribbons $\gamma_t$ in Definition \ref{def:gammat}. According to Propositions \ref{pro:curvaturapositiva} and \ref{pro:curvaturanegativa}, for sufficiently small values of $t$, the corresponding disks $\Sigma_{\gamma_t} \subset \mathbb{B}^3$ are saddle. Moreover, by Proposition \ref{pro:psiFB}, they are free boundary in $\mathbb{B}^3$. This completes the proof of Theorem \ref{thm:main}.

\bibliographystyle{amsalpha}

\vskip 0.2cm

\noindent Alberto Cerezo

\noindent Departamento de Geometría y Topología \\ Universidad de Granada (Spain) \\ \\
Instituto de Matemáticas IMUS \\ Departamento de Matemática Aplicada I \\ Universidad de Sevilla (Spain)

\noindent  e-mail: {\tt cerezocid@ugr.es}
\end{document}